\documentclass{article}
\usepackage{graphicx}
\usepackage[utf8]{inputenc}
\usepackage[T1]{fontenc}
\usepackage{amsmath,amssymb,lmodern}
\usepackage{amsthm}
\usepackage{caption}
\usepackage{color,soul}
\usepackage{fancyhdr}
\newtheorem{lemma}{Lemma}
\newtheorem{remark}{Remark}
\newtheorem{result}{Result}
\newtheorem{theorem}{Theorem}

\begin{document}

\title{Aposteriori error estimation of Subgrid multiscale stabilized finite element method for transient Stokes model}

\author{Manisha Chowdhury\thanks{ Email addresses: chowdhurymanisha8@gmail.com(M.Chowdhury) } }
      
\date{Indian Institute of Technology Kanpur \\ Kanpur, Uttar Pradesh, India}

\maketitle

\begin{abstract}
In this study, we present a novel stabilized finite element analysis for transient Stokes model. The algebraic subgrid multiscale approach has been employed to arrive at the stabilized coupled variational formulation. Derivation of the stabilized form as well as stability analysis of  it's fully discrete formulation are presented elaborately. Discrete $inf$-$sup$ condition for pressure stabilization has been proven. For the time discretization the fully implicit schemes have been used. A detailed derivation of the aposteriori error estimate for the stabilized subgrid multiscale finite element scheme has been presented. Numerical experiment has been carried out to verify theoretically established order of convergence.
\end{abstract}

\section{Introduction}
For a long period of time different numerical methods such as finite difference, finite volume and finite element methods have been employed to study the Stokes equations, widely used fluid flow mathematical model. This mixed problem does not satisfy $inf$-$sup$ condition for equal-order velocity pressure interpolation and sometimes shows pressure instability. In this regard several stabilization techniques such as the streamline-upwind Petrov–Galerkin (SUPG) \cite{RefE} -\cite{RefG}, the pressure stabilized Petrov–Galerkin (PSPG) \cite{RefH}-\cite{RefI}, the Discontinuous Galerkin \cite{RefJ}, the symmetric pressure stabilized Galerkin method \cite{RefK}, have been introduced to deal with this instability problem. 
In this paper we present Subgrid multicsale stabilized finite element analysis for transient Stokes model. Hughes in \cite{RefC} has introduced the concept of stabilized multiscale subgrid method for Helmholtz equation and further developments are going on afterwards. Generally two approaches of $SGS$ stabilized formulation, namely algebraic approach, abbreviated as $ASGS$ and orthogonal projection approach, known as $OSGS$ method,  have been studied.  $Badia$ and $Codina$ in \cite{RefB} have studied both the approaches for unified Stokes-Darcy fluid flow problem and experimentally established equally well performances of both the stabilized formulations.  In \cite{RefD} Codina presents a study on comparison of stabilized finite element methods viz. $SUPG$, $GLS$, $SGS$, $Taylor-Galerkin$ etc. for solving diffusion-convection-reaction equation and experimentally shows that $SGS$ performs well in compared to other stabilized method. Here we have derived algebraic subgrid scale ($ASGS$) stabilized finite element method for the model using continuous velocities and pressure spaces across the inter-element boundaries.Therefore it is possible to eliminate the jump terms in the variational formulation. \vspace{1mm}\\
In particular the $ASGS$ approach consists of algebraic approximation of the subscales that arise from the decomposition of the exact solution field into resolvable scale and unresolvable scale, have been used for finite element scheme stabilization. Stabilization parameters are derived following the approach in \cite{RefB} for $ASGS$ method. For time discretization fully implicit schemes have been chosen. In this study to ensure pressure stability discrete $inf-sup$ condition has been established for the choice of finite dimensional spaces and the stabilized bilinear form is shown coercive. Furthermore $aposteriori$ error estimate for the current stabilized ASGS finite element method for the unsteady Stokes model has been derived. Generally error estimation provides important information for finding out the convergence rate of the numerical method and it's dependency on the parameters present in the model. We have derived $residual$ based aposteriori error estimate which is computable and very beneficial for studying adaptivity and control of solution error. This estimation provides convergence rates with respect to space and time both and the norm defined here to find the estimation includes the standard norms involving both velocity and pressure variable. Consequently it gives a wholesome information about the convergence of the method. Numerical studies have shown the realization of theoretical order of convergence and the robustness of current stabilized $ASGS$ finite element method for Stokes system. \vspace{1mm}\\
Organisation of the paper is as follows: Section 2  introduces the model and it's weak form. Next section  presents space and time discretized variational formulation. Section 4 describes the derivation of stabilized multiscale subgrid formulation and stability analysis of the fully-discrete form. Next section has elaborately described the derivation of aposteriori error estimation for this subgrid formulation. At last section 6 contains numerical results to verify the numerical performance of the method.       

\section{ Statement of the problem}
Let $\Omega \subset \mathbb{R}^d$, d=2,3 be an open bounded domain with piecewise smooth boundary $\partial \Omega$. For the sake of simplicity in further calculations, we have considered two dimensional model, but it can be easily extended for three dimensional model.  
Let us now introduce here the system of fluid flow in $\Omega$: Find $\textbf{u}$: $\Omega$ $\times$ [0,T] $\rightarrow \mathbb{R}^2$ , $p$: $\Omega \times$ [0,T] $\rightarrow \mathbb{R}$  such that,
\begin{equation}
\begin{split}
\frac{\partial \textbf{u}}{\partial t}- \mu \Delta \textbf{u} + \bigtriangledown p & = \textbf{f} \hspace{2mm} in \hspace{2mm} \Omega \times [0,T] \\
\bigtriangledown \cdot \textbf{u} &= 0 \hspace{2mm} in \hspace{2mm} \Omega \times [0,T] \\
\textbf{u}  &= \textbf{0} \hspace{2mm} on \hspace{2mm} \partial \Omega \times [0,T] \\
\textbf{u} &= \textbf{u}_0 \hspace{2mm} at \hspace{2mm} t=0 \\
\end{split}
\end{equation}
where $(\textbf{u},p)$ is a pair of Stokes velocities and pressure respectively and $\mu$ is the Stokes dynamic viscosity and $\textbf{f}$ is the force term.\\
In operator form the above system of equations can be written as: find \textbf{U}= (\textbf{u},p) 
\begin{equation}
M\partial_t \textbf{U} + \mathcal{L}( \textbf{U}) = \textbf{F}
\end{equation}
where M, a matrix = diag(1,1,0), $\partial_t \textbf{U} = [\frac{\partial \textbf{u}}{\partial t}, \frac{\partial p}{\partial t}]^T$ \\
\[
\mathcal{L} ( \textbf{U})=
  \begin{bmatrix}
   - \mu \Delta \textbf{u} + \bigtriangledown p\\
    \bigtriangledown \cdot \textbf{u} \\
  \end{bmatrix}
 and  \hspace{1mm} \textbf{F}=
  \begin{bmatrix}
    \textbf{f} \\
    0 \\
  \end{bmatrix}
\]
Let us introduce the adjoint $\mathcal{L}^*$ of $\mathcal{L}$ as follows,
\[
\mathcal{L}^* ( \textbf{U})=
  \begin{bmatrix}
   - \mu \Delta \textbf{u}  - \bigtriangledown p\\
    -\bigtriangledown \cdot \textbf{u} \\
  \end{bmatrix}
\]
Now we impose suitable assumptions, that are necessary to conclude the results further, on the coefficients mentioned above. \vspace{1mm}\\
\textbf{(i)} $\mu$ is positive constant.\vspace{1mm}\\
\textbf{(ii)} The spaces of continuous solutions $\textbf{u}=(u_1,u_2)$ and $p$ are taken as the following:
\textbf{(a)} The velocities $u_1,u_2 \in L^{\infty}(0,T;H^2(\Omega))\bigcap$ $C^{0}(0,T; H^1(\Omega))$  \\
 \textbf{(b)} and the pressures $p$ $ \in L^{\infty}(0,T;H^1(\Omega))\bigcap C^{0}(0,T;L^2_0(\Omega)) $. \vspace{2mm}\\
 \textbf{(iii)} Additional assumptions imposed on continuous velocity solution are:  $\textbf{u}_{tt}$ and $\textbf{u}_{ttt}$ are taken to be bounded functions on $\Omega$.\vspace{1mm}\\
\textbf{Weak formulation}:
To present variational formulation let us first introduce the spaces $V_{\textbf{u}}$ and $Q$   for  velocity and pressure respectively in the following:
\begin{equation}
\begin{split}
V_{\textbf{u}} & = \{\textbf{v} \in (H^1(\Omega))^2:  \textbf{v}= \textbf{0} \hspace{1mm} on \hspace{1mm} \partial \Omega \} \\
Q & = \{ q \in L^2(\Omega): \int_{\Omega} q \hspace{1mm} d \Omega=0 \} \\
\end{split}
\end{equation}
The variational formulation of (1) is to find $\textbf{U}=(\textbf{u},p) \in V_{\textbf{u}} \times Q $ $(=V)$ such that $\forall$ $\textbf{V}=(\textbf{v},q) \in V$
\begin{equation}
\begin{split}
(M \partial_t \textbf{U}, \textbf{V})+ a_S (\textbf{u},\textbf{v})-b(\textbf{v},p) +b(\textbf{u},q)=L(\textbf{V})
\end{split}
\end{equation}
where  $(M \partial_t \textbf{U}, \textbf{V})= (\partial_t \textbf{u},\textbf{v})$ ;
$a_{S}(\textbf{u},\textbf{v})= \int_{\Omega_S} \mu \bigtriangledown \textbf{u} : \bigtriangledown \textbf{v}$ ;
$b(\textbf{u},q)=\int_{\Omega} (\bigtriangledown \cdot \textbf{u}) q$ and
$L(\textbf{V})=(\textbf{f}_1,\textbf{v})_S$\vspace{1mm}\\
The above variational formulation can be written in the following compact way: find $\textbf{U} \in V $ such that
\begin{equation}
(M \partial_t \textbf{U}, \textbf{V})+ B(\textbf{U},\textbf{V})= L(\textbf{V}) \hspace{2mm} \forall \textbf{V} \in V
\end{equation}

\section{Discrete formulations}
\subsection{Space discretisation}
 let the domain $\Omega$ be discretized into finite numbers of subdomains $\Omega_k$ for k=1,2,...,$n_{el}$, where $n_{el}$ is the total number element subdomains. Let $h_k$ be the diameter of each sub-domain $\Omega_k$. \vspace{1mm}\\
Let $ h$ = $\underset{k=1,2,...n_{el}}{max} h_k$ and $\tilde{\Omega}= \bigcup_{k=1}^{n_{el}} \Omega_k$ be the union of interior elements.\vspace{1 mm}\\
Let $V_{\textbf{u}}^h$ and $Q^h$ be finite dimensional subspaces of $V_{\textbf{u}}$ and $Q$ respectively such that
$V_{\textbf{u}}^h= \{ \textbf{v} \in V_{\textbf{u}}: \textbf{v}(\Omega_k)= \mathcal{P}^l(\Omega_k)\} $,
 and $Q^h= \{ q \in Q : q(\Omega_k)= \mathcal{P}^{l-1}(\Omega_k)\}$
 \vspace{1 mm}\\
where $\mathcal{P}^l(\Omega_k)$  denotes complete polynomial of order $l$ respectively over each $\Omega_k$ for k=1,2,...,$n_{el}$. \\
The discrete formulation is to find $\textbf{U}_h \in V_{\textbf{u}}^h \times Q^h  $ $(=V_h)$ such that
\begin{equation}
(M \partial_t \textbf{U}_h, \textbf{V}_h)+ B(\textbf{U}_h,\textbf{V}_h)= L(\textbf{V}_h) \hspace{2mm} \forall \textbf{V}_h \in V_h
\end{equation}

\subsection{Time discretisation}
For time discretization let us introduce the following uniform partition of the time interval [0,T]: for time step size $dt$= $\frac{T}{N}$, where $N$ is a positive integer, $n$-th time step $t_n= n dt$ and for given $0 \leq \theta \leq 1$,
\begin{equation}
\begin{split}
f^n & := f(\cdot , t_n) \hspace{4 mm} for \hspace{2 mm} 0 \leq n \leq N\\
f^{n,\theta} & := \frac{1}{2} (1 + \theta) f^{(n+1)} + \frac{1}{2} (1- \theta) f^n \hspace{4mm} for \hspace{2mm} 0\leq n \leq N-1 \\
dt & := t^{n+1}-t^n \\
\partial_t f^{n,\theta} & := \frac{f^{n,\theta}-f^n}{\frac{1+\theta}{2} dt}
\end{split}
\end{equation}
We see for $\theta=0$ the discretization follows Crank-Nicolson formula and for $\theta=1$ it is backward Euler discretization rule. \vspace{1mm}\\
For sufficiently smooth function $f(t)$, using the Taylor series expansion about t= $t^{n,\theta}$, we will have \vspace{1mm}\\
\begin{equation}
\begin{split}
f^{n+1} & = f(t^{n,\theta})+ \frac{(1-\theta)  dt}{2} \frac{\partial f}{\partial t}(t^{n,\theta}) + \frac{(1-\theta)^2 dt^2}{8} \frac{\partial^2 f}{\partial t^2} (t^{n,\theta}) + \mathcal{O}(dt^3)\\
f^{n} & = f(t^{n,\theta})- \frac{(1+\theta) dt}{2} \frac{\partial f}{\partial t}(t^{n,\theta}) + \frac{(1+\theta)^2 dt^2}{8} \frac{\partial^2 f}{\partial t^2}(t^{n,\theta}) + \mathcal{O}(dt^3)
\end{split}
\end{equation}
We have considered here $t^{n,\theta}- t^n= \frac{(1+\theta) dt}{2}$\\
Multiplying the above first and second sub-equations in (14) by $\frac{1+\theta}{2}$ and $\frac{1-\theta}{2}$ respectively and then adding them we will have the following\\
\begin{equation}
f^{n,\theta} = f(t^{n,\theta}) + \frac{1}{8} (1+\theta)(1-\theta) dt^2 \frac{\partial^2 f}{\partial t^2}(t^{n,\theta}) + \mathcal{O}(dt^3)
\end{equation} 
Let $\textbf{u}^{n,\theta},p^{n,\theta}$ be approximations of $\textbf{u}(\textbf{x},t^{n,\theta}), p(\textbf{x},t^{n,\theta})$ respectively. Now by Taylor series expansion \cite{RefL},we have 
\begin{equation}
\begin{split}
\frac{ \textbf{u}^{n+1}-\textbf{u}^n}{dt} & = \textbf{u}_t(\textbf{x},t^{n,\theta}) + TE_1\mid_{t=t^{n,\theta}} \hspace{5mm} \forall \textbf{x} \in \Omega
\end{split}
\end{equation}
where the truncation error $TE_1\mid_{t=t^{n,\theta}}$ $\simeq$ $TE_1^{n,\theta}$  depends upon time-derivatives of the respective variables and $dt$ \cite{RefL}.
Now applying assumption $\textbf{(iii)}$ on $\textbf{u}_{tt}$ and $\textbf{u}_{ttt}$ we will have another property as follows:
\begin{equation}
\begin{split}
\|TE_1^{n,\theta}\| & \leq  \begin{cases}
      \tilde{C}_1 dt  & if \hspace{1mm} \theta=1 \\
    \tilde{C}_2 dt^2 & if \hspace{1mm} \theta=0
      \end{cases}
\end{split}
\end{equation}
According to this rule we need to solve for $\textbf{U}^{n,\theta}_h$, $\forall \hspace{1mm} \textbf{V}_h \in V_h$
\begin{equation}
(M \partial_t \textbf{U}^{n,\theta}_h, \textbf{V}_h)+ B( \textbf{U}_h^{n,\theta}, \textbf{V}_h) = L(\textbf{V}_h) + (\textbf{TE}^{n,\theta},\textbf{V}_h)  \hspace{1mm}
\end{equation}
and the exact solution $\textbf{U}^{n,\theta}$ satisfies the above equation in the following way: $\forall \hspace{1mm} \textbf{V}_h \in V_h$
\begin{equation}
(M \partial_t \textbf{U}^{n,\theta}, \textbf{V}_h)+ B( \textbf{U}^{n,\theta}, \textbf{V}_h) = L(\textbf{V}_h) + (\textbf{TE}^{n,\theta},\textbf{V}_h)  \hspace{1mm}
\end{equation}

\section{Stabilized multiscale formulation}
Now we start deriving stabilized formulation with decomposing additively the exact solution into the resolvable scale $\textbf{U}_h$, which is a finite element solution and unresolvable scale term $\textbf{U}'$, known as subgrid scale.
\begin{equation}
\textbf{U}=\textbf{U}_h+\textbf{U}'
\end{equation}
where $\textbf{U}'$ belongs to a space $V'$ which completes $\textbf{U}_h$ in $\textbf{U}$. The test function $\textbf{V}$ can too be decomposed likewise into the components $\textbf{V}_h$ and $\textbf{V}'$. Following the classical approach we substitute (13) in (4) as follows:
\begin{equation}
\begin{split}
(M \partial_t \textbf{U}_h, \textbf{V}_h)+ B(\textbf{U}_h,\textbf{V}_h) + (M \partial_t \textbf{U}', \textbf{V}_h)+ B(\textbf{U}',\textbf{V}_h) & 
= L(\textbf{V}_h) \hspace{1mm} \forall \textbf{V}_h \in V_h\\
(M \partial_t \textbf{U}_h, \textbf{V}')+ B(\textbf{U}_h,\textbf{V}') + (M \partial_t \textbf{U}', \textbf{V}')+ B(\textbf{U}',\textbf{V}') &
= L(\textbf{V}') \hspace{1mm} \forall \textbf{V}' \in V'
\end{split}
\end{equation}
Now integrating the second sub-equation of (14) and applying suitable boundary conditions we have
\begin{equation}
\int_{\tilde{\Omega}} \textbf{V}' \cdot [M \partial_t \textbf{U}' + \mathcal{L} \textbf{U}']= \int_{\tilde{\Omega}} \textbf{V}' \cdot [\textbf{F}-M \partial_t \textbf{U}_h + \mathcal{L} \textbf{U}_h]
\end{equation}
Consideration of continuous velocities, pressure at the inter-element boundaries makes the jump term vanishes in the above equation and we obtain over each element $\Omega_k$
\begin{equation}
M \partial_t \textbf{U}' + \mathcal{L} \textbf{U}'=\textbf{F}-M \partial_t \textbf{U}_h + \mathcal{L} \textbf{U}_h= The \hspace{1mm} Residual \hspace{1mm} (R(\textbf{U}_h))
\end{equation}
along with boundary condition on $\textbf{U}'$ which is not known. Now we solve for $\textbf{U}'$ the above equation. Let us assume an approximation of the differential operator in this way: $\mathcal{L} \approx \tau_k^{-1}$, where $\tau_k$ is a matrix whose components are known as stabilization parameters \cite{RefB}. Now substituting this in (16) and applying time discretization rule we have
\begin{equation}
\frac{1}{\frac{1+\theta}{2}dt} M(\textbf{U}'-\textbf{U}^{'n}) + \tau_k^{-1}  \textbf{U}'=R(\textbf{U}_h)
\end{equation}
this implies an expression for subgrid scale term in the following
\begin{equation}
\begin{split}
\textbf{U}' & =(\frac{1}{\frac{1+\theta}{2}dt} M + \tau_k^{-1} )^{-1}\{R(\textbf{U}_h)+ \frac{1}{\frac{1+\theta}{2}dt} M \textbf{U}^{'n}\} \\
& = \tau_k' (R(\textbf{U}_h)+ \textbf{d})
\end{split}
\end{equation}
where the matrices $ \tau_k'= (\frac{1}{\frac{1+\theta}{2}dt} M + \tau_k^{-1} )^{-1}$ and $\textbf{d}= \frac{1}{\frac{1+\theta}{2}dt} M \textbf{U}^{'n}$. \vspace{1mm}\\
Now substituting the result (16)  in the first sub-problem of  (14) followed by  integrating the fourth term once and substituting the expression (18) obtained for $\textbf{U}'$, we have 
\begin{equation}
\begin{split}
(M \partial_t \textbf{U}_h, \textbf{V}_h)+ B(\textbf{U}_h,\textbf{V}_h) + \sum_{k=1}^{n_{el}} (R(\textbf{U}_h)-\tau_k^{-1} \tau_k' (R(\textbf{U}_h)+ \textbf{d}),\textbf{V}_h)_k & \\
+ \sum_{k=1}^{n_{el}}( \tau_k' (R(\textbf{U}_h)+ \textbf{d}), \mathcal{L}^* \textbf{V}_h)_k= L(\textbf{V}_h) \hspace{1mm} \forall \textbf{V}_h \in V_h
\end{split}
\end{equation}
Now expanding the residual term we have the final form of stabilized formulation in the following
\begin{equation}
(M \partial_t \textbf{U}_h, \textbf{V}_h)+ B_{ASGS}(\textbf{U}_h,\textbf{V}_h) =  L_{ASGS}(\textbf{V}_h) \hspace{1mm} \forall \textbf{V}_h \in V_h
\end{equation}
where $B_{ASGS}(\textbf{U}_h, \textbf{V}_h)= B(\textbf{U}_h, \textbf{V}_h)+ \sum_{k=1}^{n_{el}} (\tau_k'(M\partial_t \textbf{U}_h + \mathcal{L}\textbf{U}_h-\textbf{d}), -\mathcal{L}^*\textbf{V}_h)_{\Omega_k}- \sum_{k=1}^{n_{el}}((I-\tau_k^{-1}\tau_k')(M\partial_t \textbf{U}_h + \mathcal{L}\textbf{U}_h), \textbf{V}_h)_{\Omega_k}-\sum_{k=1}^{n_{el}} (\tau_k^{-1}\tau_k' \textbf{d}, \textbf{V}_h)_{\Omega_k}$ \vspace{1 mm}\\
$L_{ASGS}(\textbf{V}_h)= L(\textbf{V}_h)+ \sum_{k=1}^{n_{el}}(\tau_k' \textbf{F}, -\mathcal{L}^*\textbf{V}_h)_{\Omega_k}- \sum_{k=1}^{n_{el}}((I-\tau_k^{-1}\tau_k')\textbf{F}, \textbf{V}_h)_{\Omega_k}$  \vspace{1 mm} \\
where the stabilization parameter $\tau_k$ is in matrix form as 
\[
\tau_k= diag(\tau_{1k},\tau_{1k},\tau_{2k}) =
  \begin{bmatrix}
    \tau_{1k} I_{2 \times 2} & 0  \\
    0 & \tau_{2k}  \\
  \end{bmatrix}
\]
and 
\[
\tau_k'= (\frac{1}{dt}M+ \tau_k^{-1})^{-1} =
  \begin{bmatrix}
    \frac{\tau_{1k} dt}{dt+ \rho \tau_{1k}}I_{2 \times 2} & 0  \\
    0 & \tau_{2k}  \\
  \end{bmatrix}\\
  = diag (\tau_{1k}',\tau_{1k}',\tau_{2k}') 
\]
and the matrix $\textbf{d}$= $\sum_{i=1}^{n+1}(\frac{1}{dt}M\tau_k')^i(\textbf{F} -M\partial_t \textbf{U}_h - \mathcal{L}(\textbf{U}_h))$ =$[\textbf{d}_1,d_2]^T$ \vspace{1 mm}\\
It can be easily observed that $d_2$ is always 0 due to the matrix M. Now the expressions of stabilization parameters \cite{RefB} are $\tau_{1k}=\frac{h^2}{c_1 \mu}$ and $\tau_{2k}=\frac{c_2}{\tau_{1k}}$, where $c_1,c_2$ are constants. \vspace{1 mm}\\
According to the time discretization rule we need to solve for $\textbf{U}^{n,\theta}_h$, $\forall \hspace{1mm} \textbf{V}_h \in V_h$
\begin{equation}
(M \partial_t \textbf{U}^{n,\theta}_h, \textbf{V}_h)+ B_{ASGS}(\textbf{U}_h^{n,\theta}, \textbf{V}_h) = L_{ASGS}(\textbf{V}_h) + (\textbf{TE}^{n,\theta},\textbf{V}_h)  \hspace{1mm}
\end{equation}

\subsection{Stability analysis of fully-discrete stabilized form}
Let us first mention discrete $inf$-$sup$ condition for discrete formulation.
\begin{lemma}{\textbf{Discrete inf-sup condition}:}
There exists a constant $\beta > 0$, independent of $h$ , such that 
\begin{center}
$\underset{q_h \in Q_h}{inf} \underset{\textbf{v}_h \in V^d_h}{sup} \frac{ \mid b(q_h,\textbf{v}_h) \mid}{\|\textbf{v}_h\|_1 \|q_h\|_0} \geq \beta$
\end{center}
\end{lemma}
\begin{proof}
Let $\pi_h$ be $L^2$ projection on $V_h$ satisfying the following relations:\vspace{1mm}\\
(i)Stability estimates \cite{RefD2} for $u \in H^1(\Omega)$, $\|\pi_h u\|_0 \leq C \|u\|_0$ and $\| \pi_h u\|_1 \leq C \|u\|_1$  \\
(ii) Interpolation estimate \cite{RefD3} for $\textbf{v} \in (H^1(\Omega))^d$ $\| \bigtriangledown \cdot (\textbf{v}- \pi_h \textbf{v}) \| \leq \bar{C} h^r \| \bigtriangledown \cdot \textbf{v} \|_r$ for $0 \leq r \leq 1.$ \vspace{1mm}\\
Let $q_h \in Q_h$. From \cite{RefD1} there exists $\textbf{v} \in (H_0^1(\Omega))^d$ such that $\bigtriangledown \cdot \textbf{v} = q_h$ and $\|\textbf{v}\|_1 \leq C_1 \|q_h\|_0$
\begin{equation}
\begin{split}
\|q_h\|_0^2  = (q_h, \bigtriangledown \cdot \textbf{v})  & = (q_h, \bigtriangledown \cdot \textbf{v}- \bigtriangledown \cdot \pi_h \textbf{v}) + (q_h, \bigtriangledown \cdot \pi_h \textbf{v}) \\
& \leq \|q_h\|_0 \|\bigtriangledown \cdot (\textbf{v}- \textbf{v}_h)\|_0 + \mid b(q_h, \pi_h \textbf{v}) \mid \\
& \leq \|q_h\|_0 \|\textbf{v}- \textbf{v}_h\|_1 +  \mid b(q_h, \pi_h \textbf{v}) \mid \\
& \leq \|q_h\|_0 \bar{C} \|\textbf{v}\|_1 +  \mid b(q_h, \pi_h \textbf{v}) \mid \\
& \leq \bar{C} C_1 \|q_h\|_0^2 + \mid b(q_h, \pi_h \textbf{v}) \mid \\
\end{split}
\end{equation}
which implies
\begin{equation}
\begin{split}
 \mid b(q_h, \pi_h \textbf{v}) \mid & \geq (1-\bar{C} C_1) \|q_h\|_0^2 \\
& \geq \frac{(1-\bar{C} C_1)}{C_1} \|q_h\|_0 \|\textbf{v}\|_1 \\
& \geq  \frac{(1-\bar{C} C_1)}{C_1 C} \|q_h\|_0 \|\pi_h \textbf{v}\|_1
\end{split}
\end{equation}
Choose $\beta = \frac{(1-\bar{C} C_1)}{C_1 C}$ provided $\bar{C} C_1 < 1$
\end{proof}
For studying stability analysis of fully discrete formulation we have considered backward Euler time discretization rule for instance and arrived at the following: for given $\textbf{U}^n_h \in V_h$ find $\textbf{U}^{n+1}_h \in V_h$ such that $\forall \textbf{V}_h \in V_h$
\begin{multline}
(\frac{M}{dt} \textbf{U}^{n+1}_h, \textbf{V}_h) + B(\textbf{U}_h^{n+1},\textbf{V}_h) - \sum_{k=1}^{n_{el}} ((I-\tau^{-1}_k \tau_k')(\frac{M}{dt}  \textbf{U}^{n+1}_h+\mathcal{L}(\textbf{U}_h^{n+1})),\textbf{V}_h)_{\Omega_k}\\
+ \sum_{k=1}^{n_{el}}(\tau_k'(\frac{M}{dt} \textbf{U}^{n+1}_h+\mathcal{L}(\textbf{U}_h^{n+1})),  -\mathcal{L}^*( \textbf{V}_h))_{\Omega_k} = L_{ASGS}(\textbf{V}_h) +(\frac{M}{dt} \textbf{U}^{n}_h, \textbf{V}_h)- \\
\sum_{k=1}^{n_{el}} ((I-\tau^{-1}_k \tau_k')(\frac{M}{dt}  \textbf{U}^{n+1}_h,\textbf{V}_h)_{\Omega_k}+\sum_{k=1}^{n_{el}}(\tau^{-1}_k \tau_k' \textbf{d},\textbf{V}_h)_{\Omega_k}-\sum_{k=1}^{n_{el}}(\tau_k' \textbf{d}, \mathcal{L}^*(\textbf{V}_h))_{\Omega_k} \\
+\sum_{k=1}^{n_{el}}(\tau_k'(\frac{M}{dt} \textbf{U}^{n}_h, -\mathcal{L}^*(\textbf{V}_h))_{\Omega_k} 
\end{multline}
Now dropping the superscripts let us denote the unknown part in the following
\begin{multline}
\bar{B}_{ASGS}(\textbf{U}_h, \textbf{V}_h;dt):= B( \textbf{U}_h, \textbf{V}_h)  - \sum_{k=1}^{n_{el}} ((I-\tau^{-1}_k \tau_k')(\frac{M}{dt}  \textbf{U}_h+\mathcal{L}(\textbf{U}_h)),\textbf{V}_h)_{\Omega_k}\\
+ (\frac{M}{dt} \textbf{U}_h, \textbf{V}_h)+ \sum_{k=1}^{n_{el}}(\tau_k'(\frac{M}{dt} \textbf{U}_h+\mathcal{L}(\textbf{U}_h)),  -\mathcal{L}^*( \textbf{V}_h))_{\Omega_k}
\end{multline}

\begin{theorem}
For regular partitions satisfying inverse inequalities and assuming a condition $ dt>\bar{C} h^2$ there exists positive parameters $C_i$s (for i=1,2,3) depending upon $h$ such that \\
$\bar{B}_{ASGS}( \textbf{U}_h, \textbf{U}_h;dt) \geq C_1 \|\textbf{u}_h\|^2+ C_2 \|\bigtriangledown \textbf{u}_h\|^2+ C_3 \|\bigtriangledown p_h\|^2 $
\end{theorem}
\begin{proof}
Expanding the terms of (20) we have
\begin{multline}
B_{ASGS}( \textbf{U}_h, \textbf{U}_h;dt)= \frac{1}{dt} \|\textbf{u}_h\|^2 + a_{S}(\textbf{u}_h, \textbf{u}_h) 
- \sum_{k=1}^{n_{el}}(\frac{\tau_{1k}}{dt + \tau_{1k}} I_{2 \times 2}( \frac{\textbf{u}_h}{dt} - \mu \Delta \textbf{u}_h+ \\
\bigtriangledown p_h), \textbf{u}_h)_{\Omega_k}+ 
 \sum_{k=1}^{n_{el}}(\frac{ \tau_{1k} dt}{dt +  \tau_{1k}} I_{2 \times 2}( \frac{\textbf{u}_h}{dt} - \mu \Delta \textbf{u}_h+ \bigtriangledown p_h),  \mu \Delta \textbf{u}_h+ \bigtriangledown p_h)_{\Omega_k} \\
\end{multline}
let us look at the terms separately. 
 $a_{S}(\textbf{u}_h, \textbf{u}_h) = \mu \| \bigtriangledown \textbf{u}_h \|^2$ \vspace{1mm} \\
Now applying Cauchy-Schwarz inequality on the next terms we have \vspace{1mm}\\
$( \frac{  \tau_{1k} }{ dt +  \tau_{1k} } I_{2 \times 2} \frac{\textbf{u}_h}{dt}, \textbf{u}_h)_{\Omega_k} \leq \frac{1}{dt} (\sum_{k=1}^{n_{el}}\frac{  \mid \tau_k \mid}{\mid dt + \tau_k \mid }) \|\textbf{u}_h\|_k^2  $ \vspace{1mm}\\
 and applying inverse inequalities in the following intermediate steps we have
\begin{multline}
( \frac{  \tau_{1k} dt}{ dt +  \tau_{1k} } I_{2 \times 2} \frac{\textbf{u}_h}{dt}, \mu \Delta \textbf{u}_h+ \bigtriangledown p_h)_{\Omega_k}  
\leq   \frac{  \tau_{1k} }{ dt +  \tau_{1k} } ( \textbf{u}_h,  \mu \Delta \textbf{u}_h+ \bigtriangledown p_h)_{\Omega_k}  \\
\leq   \frac{  \tau_{1k} }{ dt +\tau_{1k} } \{  \mu \|\textbf{u}_h\|_k \|\Delta \textbf{u}_h\|_k + \|\textbf{u}_h\|_k \|\bigtriangledown p_h\|_k \}\\
\leq    \frac{  \tau_{1k} }{ dt +  \tau_{1k} }  \{\epsilon_1 \|\textbf{u}_h\|_k^2 + \frac{\mu^2 C_I^2}{2 \epsilon_1 h^2} \|\bigtriangledown \textbf{u}_h\|_k^2+ \frac{1}{2 \epsilon_1} \|\bigtriangledown p_h\|_k^2\}
\end{multline}
Now the next term similarly as above
\begin{multline}
( \frac{ \tau_{1k}}{dt+  \tau_{1k}} I_{2 \times 2}( - \mu \Delta \textbf{u}_h+ \bigtriangledown p_h), \textbf{u}_h)_{\Omega_k} \\
\leq    \frac{ \tau_{1k} }{ dt + \tau_{1k} }  \{\epsilon_1 \|\textbf{u}_h\|_k^2 + \frac{\mu^2 C_I^2}{2 \epsilon_1 h^2} \|\bigtriangledown \textbf{u}_h\|_k^2+ \frac{1}{2 \epsilon_1} \|\bigtriangledown p_h\|_k^2\}
\end{multline}
and
\begin{multline}
(\frac{ \tau_{1k} dt}{dt + \rho \tau_{1k}} I_{2 \times 2}( - \mu \Delta \textbf{u}_h+ \bigtriangledown p_h), \mu \Delta \textbf{u}_h+ \bigtriangledown p_h)_{\Omega_k} \hspace{2mm} \\
\geq \frac{ \tau_{1k} dt}{dt + \rho \tau_{1k}} \{ \| \bigtriangledown p_h\|_k^2- \mu^2 \| \Delta \textbf{u}_h\|_k^2 \} 
\geq \frac{ \tau_{1k} dt}{dt + \rho \tau_{1k}} \{\| \bigtriangledown p_h\|_k^2- \frac{\mu^2 C_I^2}{h^2} \| \bigtriangledown \textbf{u}_h\|_k^2 \} \hspace{55mm}
\end{multline}
Now combining all the results
\begin{equation}
\begin{split}
\bar{B}_{ASGS}(\textbf{U}_h,\textbf{U}_h;dt) & \geq C_1 \|\textbf{u}_h\|^2 + C_2 \|\bigtriangledown \textbf{u}_h\|^2 + C_3 \|\bigtriangledown p_h\|^2
\end{split}
\end{equation}
where $C_1 = \frac{ (1-2 \epsilon_1 \tau_1)}{dt+ \tau_1}=\frac{\tau_1^{-1}-2 \epsilon_1}{dt \tau_1^{-1}+ 1}=\frac{c_1 \frac{\mu}{h^2}-2 \epsilon_1}{dt \tau_1^{-1}+1}= \frac{c_1 \mu -2 \epsilon_1 h^2}{h^2(dt \tau_1^{-1}+1)}$\vspace{1mm}\\
$C_2=\tau_1[\frac{\mu^2}{h^2} (c_1 -\frac{ C_I^2}{\epsilon_1(dt+ \tau_1)}-\frac{ C_I^2 dt}{(dt+ \tau_1)} )] \geq \tau_1 \bar{C} \tau_1^{-1} =\bar{C}_1$ \vspace{1mm} \\
$C_3= \frac{\tau_1}{dt+ \tau_1}(dt-\frac{1}{\epsilon_1})=\frac{(dt- \epsilon_1^{-1})}{dt \tau_1^{-1}+1}$ \vspace{1mm}\\
Choose arbitrary $\epsilon_1$ in such a way that $ dt > \epsilon_1^{-1}>\bar{C} h^2$ holds and therefore $C_1$ and $C_2$ can be made positive. $\bar{C}_1$ is positive constants by choice of stabilization parameters.
\end{proof}

\section{Error estimation}
We start this section with introducing the projection operator corresponding to each unknown variable followed by notation of error and it's component wise splitting. Later we derive $aposteriori$ error estimates.

\subsection{Projection operators : Error splitting}
Let us introduce the projection operator for each of these error components.\vspace{1 mm}\\
(I)For any $\textbf{u} \in H^2(\Omega) \times H^2(\Omega) $ we assume that there exists an interpolation $I^h_{\textbf{u}}:  H^2(\Omega) \times H^2(\Omega) \longrightarrow  V_{\textbf{u}}^h $ satisfying $b(\textbf{u}-I^h_{\textbf{u}}\textbf{u}, q_h)=0$ \hspace{2mm} $\forall q_h \in Q_h$ \vspace{2mm}\\
(II) Let $I^h_p: H^1(\Omega) \longrightarrow Q_s^h$ be the $L^2$ orthogonal projection given by \\ $\int_{\Omega}(p-I^h_pp)q_h=0$  \hspace{1mm} $\forall q_h \in Q_h$ and for any $p \in H^1(\Omega)$ \vspace{2mm}\\
Let $\textbf{e}=(e_{\textbf{u}},e_p)$ denote the error where the components are $e_{\textbf{u}}=(e_{u1},e_{u2})= (u_1-u_{1h}, u_2-u_{2h})$ and  $e_p= (p-p_h)$. Now each component of the error can be split into two parts interpolation part, $E^I$ and auxiliary part, $E^A$ as follows: \vspace{1mm}\\
$e_{u1}=(u_1-u_{1h})=(u_1-I^h_{u1}u_1)+(I^h_{u1}u_1-u_{1h})= E^{I}_{u1}+ E^{A}_{u1}$ \vspace{1mm}\\
Similarly, $e_{u2}=E^{I}_{u2}+ E^{A}_{u2}$ and
$e_{p}=E^{I}_{p}+ E^{A}_{p}$. \vspace{2mm}\\
At this point let us mention the standard \textbf{interpolation estimation} result \cite{RefA} in the following: for any exact solution with regularity upto (m+1)
\begin{equation}
\|v-I^h_v v\|_l = \|E^I_v\|_l \leq C(p,\Omega) h^{m+1-l} \|v\|_{m+1} 
\end{equation}
where l ($\leq m+1$) is a positive integer and C is a constant depending on m and the domain. For l=0 and 1 it  implies standard $L^2(\Omega)$ and $H^1(\Omega)$ norms respectively. For simplicity we will use $\| \cdot \|$ instead of $\| \cdot \|_0$ to denote $L^2(\Omega)$ norm.
Now we put some results using the properties of projection operators and these results will be used in error estimations.
\begin{result}
\begin{equation}
(\frac{\partial}{\partial t} E^{I,n}_{\textbf{u}}, \textbf{v}_h)=0 \hspace{2mm} \textbf{v}_h \in V_h
\end{equation}
\end{result}

\begin{result}
For any given auxiliary error $E^{A,n}$ and unknown $E^{A,n+1}$
\begin{equation}
(\frac{\partial}{\partial t} E^{A,n}, E^{A,n,\theta}) \geq  \frac{1}{2 dt} (\|E^{A,n+1}\|^2- \|E^{A,n}\|^2)
\end{equation}
\end{result}

\subsection{Aposteriori error estimation}
In this section we are going to derive residual based $aposteriori$ error estimation. Before deriving error estimations let us define  required norms the estimations. Let us consider the space $\tilde{\textbf{V}} $ := $L^2(0,T; V_s)\bigcap L^{\infty}(0,T; Q_s)$ and it's associated norm is denoted by $\tilde{\textbf{V}}$-norm. For the functions $g_1, g_2, g_3$ belonging to the spaces $L^2(0,T; L^2(\Omega))$, $L^2(0,T; H_0^1(\Omega))$,  $\tilde{\textbf{V}}$ respectively norms over these spaces, abbreviated as $L^2(L^2)$, $L^2(H^1)$, $\tilde{\textbf{V}}$ are defined in the following
\begin{equation}
\begin{split}
\|g_1\|_{L^2(L^2)}^2 &= \sum_{n=0}^{N-1} \int_{t^n}^{t^{n+1}} \int_{\Omega} \mid g_1^{n,\theta} \mid^2 dt \\
\|g_2\|_{L^2(H^1)}^2 &= \sum_{n=0}^{N-1} \int_{t^n}^{t^{n+1}} (\int_{\Omega} \mid g_2^{n,\theta} \mid^2  +  \int_{\Omega} \mid \frac{\partial g_2}{\partial x}^{n,\theta} \mid^2  +  \int_{\Omega} \mid \frac{\partial g_2}{\partial y}^{n,\theta} \mid^2 )dt \\
\|g_3\|_{\tilde{\textbf{V}}}^2 & = \underset{0\leq n \leq N}{max} \|g_3^n\|^2 + \|g_3\|_{L^2(H^1)}^2\\
\end{split}
\end{equation} 
\begin{theorem} 
For computed velocity $\textbf{u}_h$ and pressure $p_h$ belonging to $V_{\textbf{u}}^h$ and $Q_h$ satisfying (32)-(33), assume $dt$ is sufficiently small and positive, and sufficient regularity of exact solution in equation (1). Then there exists a constant $\bar{C}$, independent of $\textbf{u},p$ and depending on the residual such that
\begin{equation}
\|\textbf{u}-\textbf{u}_h\|^2_{\tilde{\textbf{V}}} + \|p-p_h\|_{L^2(L^2)}^2  \leq \bar{C}(\textbf{R}) (h^2+ dt^{2r})
\end{equation}
where $\textbf{R}$ is the residual vector and
\begin{equation}
    r=
    \begin{cases}
      1, & \text{if}\ \theta=1 \\
      2, & \text{if}\ \theta=0
    \end{cases}
  \end{equation}
\end{theorem}
\begin{proof}
We estimate $aposteriori$ error by dividing the procedure into two parts. In the first part we find error bound corresponding to $velocity$ and $concentration$ followed by the second part estimating error associated with the $pressure$ term. Let us first introduce the residual vector corresponding to each equations 
\[
\textbf{R}=
  \begin{bmatrix}
 \textbf{f}-\{\frac{\partial \textbf{u}_h}{\partial t}- \mu \Delta \textbf{u}_h + \bigtriangledown p_h\} \\
    -\bigtriangledown \cdot \textbf{u}_h 
  \end{bmatrix}
   = 
  \begin{bmatrix}
 \textbf{R}_1\\
  R_2 
  \end{bmatrix}
\]
\textbf{First part}: We have $\forall \textbf{V} \in V$
\begin{equation}
\mu \mid \textbf{v} \mid_1^2  \leq B(\textbf{V},\textbf{V})= a_{S}(\textbf{v},\textbf{v})
\end{equation} 
Since $\textbf{e} \in V$ we substitute the errors $e_{\textbf{u}}$ into the above relation and adding few terms in both sides we have
\begin{multline}
 (\frac{\partial e_{\textbf{u}}}{\partial t},e_{\textbf{u}})+\mu_l \|e_{\textbf{u}}\|_1^2 
 \leq   (\frac{\partial e_{\textbf{u}}}{\partial t},e_{\textbf{u}})+a_{S}(e_\textbf{u},e_\textbf{u})
 +b(e_\textbf{u},e_p)-b(e_\textbf{u},e_p)+ \mu_l \|e_{\textbf{u}}\|^2
\end{multline} 
Now first we will find a lower bound of $LHS$ and then upper bound for $RHS$ and finally combining them we will get $aposteriori$ error estimate. Applying (33) on the first term of $LHS$ we have the following relations
\begin{equation}
\begin{split}
(\frac{e_{\textbf{u}}^{n+1}-e_{\textbf{u}}^n}{dt},e_{\textbf{u}}^{n,\theta}) & \geq \frac{1}{2 dt}(\|e_{\textbf{u}}^{n+1}\|^2-\|e_{\textbf{u}}^n\|^2) \\
\end{split}
\end{equation}
Hence
\begin{equation}
 \frac{1}{2 dt}(\|e_{\textbf{u}}^{n+1}\|^2-\|e_{\textbf{u}}^n\|^2)+\mu \|e_{\textbf{u}}^{n,\theta}\|_1^2  \leq LHS \leq RHS
\end{equation}
Now our job is to find upper bound for $RHS$ and to reach at the desired estimates let us divide it into two broad parts by splitting errors in each of the terms in the following way:
\begin{equation}
\begin{split}
RHS &= [(\frac{e_{\textbf{u}}^{n+1}-e_{\textbf{u}}^n}{dt},E_{\textbf{u}}^{I,n,\theta}) +  a_{S}(e^{n,\theta}_\textbf{u},E^{I,n,\theta}_\textbf{u}) + b(e^{n,\theta}_\textbf{u},E^{I,n,\theta}_{p})-  b(E^{I,n,\theta}_\textbf{u},e^{n,\theta}_p)  ] 
\\
& \quad + [(\frac{e_{\textbf{u}}^{n+1}-e_{\textbf{u}}^n}{dt},E_{\textbf{u}}^{A,n,\theta}) + a_{S}(e^{n,\theta}_\textbf{u},E^{A,n,\theta}_\textbf{u})+ b(e^{n,\theta}_\textbf{u} ,E^{A,n,\theta}_{p}) \\
& \quad -  b(E^{A,n,\theta}_\textbf{u},e^{n,\theta}_p) ]+ \mu \|e_{\textbf{u}}^{n,\theta}\|^2 \\
&= RHS^I + RHS^A  + \mu \|e_{\textbf{u}}^{n,\theta}\|^2\\
\end{split}
\end{equation}
Our aim is to bring residual into context and for this purpose $RHS^I$ involving interpolation error terms can be estimated as follows: we have $\forall$ $\textbf{V} \in V$
\begin{multline}
 (\frac{e_{\textbf{u}}^{n+1}- e_{\textbf{u}}^{n}}{dt}, \textbf{v})+ a_{S}(e^{n,\theta}_{\textbf{u}},\textbf{v})-b(\textbf{v}, e_p^{n,\theta}) + (\bigtriangledown \cdot e_\textbf{u}^{n,\theta},q)
= \int_{\Omega} \textbf{R}_1^{n,\theta} \cdot \textbf{v} + \int_{\Omega} R_2^{n,\theta} q
\end{multline}
Now substituting $\textbf{v} ,q$ in the above expressions by $E^{I,n,\theta}_{\textbf{u}}, E^{I,n,\theta}_{p}$ respectively, we have the $RHS^I$ as,
\begin{equation}
\begin{split}
RHS^I & = \int_{\Omega} (\textbf{R}_1^{n,\theta} \cdot E^{I,n,\theta}_{\textbf{u}}+ R_2^{n,\theta} E^{I,n,\theta}_{p})\\
& \leq h^2 \{ \|\textbf{R}_1^{n,\theta}\| (\frac{1+\theta}{2} \| \textbf{u}^{n+1} \|_2 + \frac{1-\theta}{2} \| \textbf{u}^n \|_2) + C_2 \|R_2^{n,\theta}\| (\frac{1+\theta}{2} \|p^{n+1}\|_1  \\
& \quad  + \frac{1-\theta}{2} \|p^n\|_1) \}\\
& \leq h^2( \bar{C}_1 \|\textbf{R}_1^{n,\theta}\| + \bar{C}_2 \|R_2^{n,\theta}\|) \\
\end{split}
\end{equation}
The parameters $\bar{C}_i$, for i=1,2,3,4, are coming from imposing assumption  \textbf{(iv)}. Now we are going to estimate of $RHS^A$. For that we employ $subgrid$ formulation (20). Subtracting (20) from the variational finite element formulation satisfied by the exact solution we have $\forall \textbf{V}_h \in V_h $
\begin{multline}
 (\frac{e_{\textbf{u}}^{n+1}-e_{\textbf{u}}^n}{dt}, \textbf{v}_h) 
+a_{S}(e_{\textbf{u}}^{n,\theta},\textbf{v}_h)- b(\textbf{v}_h, e^{n,\theta}_p)+ b(e_{\textbf{u}}^{n,\theta}, q_h) \\
 = \sum_{k=1}^{n_{el}} \{(\tau_k'(\textbf{R}^{n,\theta}+\textbf{d}), -\mathcal{L}^* \textbf{V}_h)_{\Omega_k} - ((I-\tau_k^{-1}\tau_k) \textbf{R}^{n,\theta}, \textbf{V}_h)_{\Omega_k} + (\tau_k^{-1}\tau_k \textbf{d}, \textbf{V}_h)_{\Omega_k} \} \\
  +(\textbf{TE}^{n,\theta}, \textbf{v}_{h})\\
\end{multline} 
where the column vector $\textbf{d}$ = $[\textbf{d}_1,d_2]^T$= $\sum_{i=1}^{n+1}(\frac{1}{dt}M\tau_k')^i(\textbf{F} -M\partial_t \textbf{U}_h - \mathcal{L}\textbf{U}_h)=\sum_{i=1}^{n+1}(\frac{1}{dt}M\tau_k')^i \textbf{R}$. Hence we have the components $\textbf{d}_1= (\sum_{i=1}^{n+1}(\frac{1}{dt} \tau_1')^i) I_{d \times d} \textbf{R}_1^{n,\theta}$ and $d_2=0$.  \vspace{1mm} \\ 
Now expanding the terms in (44) further and substituting $\textbf{V}_h$ by $(E^{A,n,\theta}_{\textbf{u}},  E^{A,n,\theta}_{p})$ in the above equation we have $RHS^A$ as follows 
\begin{equation}
\begin{split}
RHS^A & = \sum_{k=1}^{n_{el}}[( \tau_1' I_{2 \times 2} \{\textbf{R}_1^{n,\theta}+\textbf{d}_1 \}, \mu \Delta E^{A,n,\theta}_{\textbf{u}}+ \bigtriangledown E^{A,n,\theta}_{p})_{\Omega_k} + \tau_2'(R_2^{n,\theta}, \bigtriangledown \cdot E^{A,n,\theta}_{\textbf{u}})_{\Omega_k}  \\
& \quad  + ((1-\tau_1^{-1}\tau_1') I_{2 \times 2} \textbf{R}_1^{n,\theta}, E^{A,n,\theta}_{\textbf{u}})_{\Omega_k} + (\tau_1^{-1}\tau_1'I_{2 \times 2}\textbf{d}_1,E^{A,n,\theta}_{\textbf{u}})_{\Omega_k}] + \\
& \quad (\textbf{TE}_1^{n,\theta},E^{A,n,\theta}_{\textbf{u}})  
\end{split}
\end{equation}
Now we estimate each term separately. Before going to further calculations let us mention an important observation: By the virtue of the choices of the finite element spaces $V_{\textbf{u}}^h$ and $Q^h$, we can clearly say that $\Omega_k$ every function belonging to that spaces and their first and second order derivatives all are bounded  over each element sub-domain. We can always find positive finite real numbers to bound each of the functions over element sub-domain. Applying Cauchy-Schwarz inequality followed by this observation on the following terms of (45) we have
\begin{equation}
\begin{split}
 \sum_{k=1}^{n_{el}}( \tau_1' I_{2 \times 2} \textbf{R}_1^{n,\theta}, \mu \Delta E^{A,n,\theta}_{\textbf{u}}+ \bigtriangledown E^{A,n,\theta}_{p})_{\Omega_k}  &  \leq \frac{\mid \tau_1 \mid T}{T_0-  C_{\tau_1}} (\sum_{k=1}^{n_{el}} D_{B_{1k}})\|\textbf{R}_1^{n,\theta}\|  \\
\sum_{k=1}^{n_{el}}((1-\tau_1^{-1} \tau_1')I_{2 \times 2} \textbf{R}_1^{n,\theta}, E^{A,n,\theta}_{\textbf{u}})_{\Omega_k} & \leq \frac{ \mid \tau_1 \mid}{T_0- C_{\tau_1}} (\sum_{i=1}^2 \sum_{k=1}^{n_{el}}B_{1k}^i) \|\textbf{R}_1^{n,\theta}\| \\
\end{split}
\end{equation} 
where the constants $D_{B_{1k}}$ and $B_{1k}^i$ appear due to imposing bounds on  $ \Delta E^{A,n,\theta}_{\textbf{u}}$, $\bigtriangledown E^{A,n,\theta}_{p}$ and $E^{A,n,\theta}_{\textbf{u}}$ over each sub-domain $\Omega_k$ and $C_{\tau_1}$ is upper bound on $\tau_1$. Now the terms containing the components of $\textbf{d}$ can be estimated in the following way:
\begin{equation}
\begin{split}
&\sum_{k=1}^{n_{el}} (\tau_1' I_{2 \times 2} \textbf{d}_1, \mu \Delta E^{A,n,\theta}_{\textbf{u}}+ \bigtriangledown E^{A,n,\theta}_{p})_{\Omega_k} \\
& = \sum_{k=1}^{n_{el}}( \tau_1' \{\sum_{i=1}^{n+1}(\frac{1}{dt} \tau_1')^i\} I_{2 \times 2}\textbf{R}_1^{n,\theta}, \mu \Delta E^{A,n,\theta}_{\textbf{u}}+ \bigtriangledown E^{A,n,\theta}_{p})_{\Omega_k} \\
& \leq \sum_{k=1}^{n_{el}}(\tau_1' \{\sum_{i=1}^{\infty}(\frac{1}{dt} \tau_1')^i\} I_{2 \times 2} \textbf{R}_1^{n,\theta},\mu \Delta E^{A,n,\theta}_{\textbf{u}}+ \bigtriangledown E^{A,n,\theta}_{p})_{\Omega_k} \\
& = \sum_{k=1}^{n_{el}}( \frac{\tau_1^2}{dt+ \tau_1} I_{2 \times 2} \textbf{R}_1^{n,\theta}, \mu \Delta E^{A,n,\theta}_{\textbf{u}}+ \bigtriangledown E^{A,n,\theta}_{p})_{\Omega_k}  \leq \frac{\mid \tau_1 \mid C_{\tau_1}}{T_0-C_{\tau_1}} (\sum_{k=1}^{n_{el}} D_{B_{1k}}) \|\textbf{R}_1^{n,\theta}\| \\
& \sum_{k=1}^{n_{el}} (\tau_1^{-1}\tau_1' I_{2 \times 2} \textbf{d}_1,E^{A,n,\theta}_{\textbf{u}})_{\Omega_k}  \leq \frac{ \mid \tau_1 \mid}{T_0-  C_{\tau_1}} (\sum_{i=1}^2 \sum_{k=1}^{n_{el}} B_{1k}^i) \|\textbf{R}_1^{n,\theta}\|
\end{split}
\end{equation}
 Now we estimate the remaining terms as follows:
\begin{equation}
\begin{split} 
 \sum_{k=1}^{n_{el}} \tau_2'(R_2^{n,\theta}, \bigtriangledown \cdot E^{A,n,\theta}_{\textbf{u}})_{\Omega_k} & =  \sum_{k=1}^{n_{el}} \tau_2' (R_2^{n,\theta}, \bigtriangledown \cdot e^{n,\theta}_{\textbf{u}})_{\Omega_k} -  \sum_{k=1}^{n_{el}} \tau_2' (R_2^{n,\theta}, \bigtriangledown \cdot E^{I,n,\theta}_{\textbf{u}})_{\Omega_k}  \\
& =  \sum_{k=1}^{n_{el}} \tau_2'  (\bigtriangledown \cdot e^{n,\theta}_{\textbf{u}}, \bigtriangledown \cdot e^{n,\theta}_{\textbf{u}})_{\Omega_k} -  \sum_{k=1}^{n_{el}} \tau_2' ( \bigtriangledown \cdot e^{n,\theta}_{\textbf{u}}, \bigtriangledown \cdot E^{I,n,\theta}_{\textbf{u}})_{\Omega_k}  \\
 \end{split}
\end{equation}
Applying Cauchy-Schwarz inequality followed by Young's inequality in the following steps we have:
\begin{equation}
\begin{split} 
& \leq  C_{\tau_2} (\sum_{i=1}^2 \|\frac{\partial e^{n,\theta}_{ui}}{\partial x_i}\|)^2 +  C_{\tau_2} (\sum_{i=1}^2 \|\frac{\partial e^{n,\theta}_{ui}}{\partial x_i}\|) (\sum_{i=1}^2 \|\frac{\partial E^{I,n,\theta}_{ui}}{\partial x_i}\| )  \\
 & \leq 2 C_{\tau_2} \sum_{i=1}^2 \|\frac{\partial e^{n,\theta}_{ui}}{\partial x_i}\|^2 + \epsilon_1' C_{\tau_2} \sum_{i=1}^2 \|\frac{\partial e^{n,\theta}_{ui}}{\partial x_i}\|^2  +  \frac{C_{\tau_2}}{\epsilon_1'} \sum_{i=1}^2 \|\frac{\partial E^{I,n,\theta}_{ui}}{\partial x_i}\|^2   \\
& \leq C_{\tau_2}[ (2+ \epsilon_1')  \sum_{i=1}^2 \| e^{n,\theta}_{ui}\|_1^2  + \frac{h^2}{\epsilon_1'} \sum_{i=1}^2 (\frac{1+\theta}{2} \|u_i^{n+1}\|_2+ \frac{1-\theta}{2} \|u_i^{n}\|_2)^2 ]  \\
& \leq   (2+ \epsilon_1') C_{\tau_2} \| e^{n,\theta}_{\textbf{u}}\|_1^2 + h^2 \frac{C_{\tau_2}}{\epsilon_1'} \bar{C}_5 
 \end{split}
\end{equation}
where the parameter $\bar{C}_5$ comes for applying assumption $\textbf{(iv)}$. Now the estimation of terms involving trancation is in the following:
\begin{equation}
\begin{split}
(\textbf{TE}^{n,\theta}, E^{A,n,\theta}_{\textbf{u}}) & = (\textbf{TE}^{n,\theta}, e^{n,\theta}_{\textbf{u}}) - (\textbf{TE}^{n,\theta}, E^{I,n,\theta}_\textbf{u}) \\
& \leq \frac{1}{\epsilon_2'} \|\textbf{TE}^{n,\theta}\|^2 + \frac{\epsilon_2'}{2} (\|e^{n,\theta}_{\textbf{u}}\|^2 + \|E^{I,n,\theta}_{\textbf{u}} \|^2) \\
& \leq  \frac{1}{\epsilon_2'} \|\textbf{TE}^{n,\theta}\|^2 +  \frac{\epsilon_2'}{2} \{ \|e^{n,\theta}_{\textbf{u}}\|^2 + h^4 (\frac{1+\theta}{2} \|\textbf{u}^{n+1}\|_2 + \frac{1-\theta}{2} \|\textbf{u}^n\|_2)^2 \} \\
& \leq \frac{1}{\epsilon_2'} \|\textbf{TE}^{n,\theta}\|^2 +  \frac{\epsilon_2'}{2} \|e^{n,\theta}_{\textbf{u}}\|^2_1 +  h^4 \frac{\epsilon_2'}{2} \bar{C}_5
\end{split}
\end{equation}
and this completes estimating all the terms of $RHS^A$. Now applying $Poincare$ inequality on last term in $RHS$ in (41) we have
\begin{equation}
\mu \|e_{\textbf{u}}^{n,\theta}\|^2 \leq \mu C_P \mid e_{\textbf{u}}^{n,\theta} \mid_1^2 \leq \mu C_P \|e_{\textbf{u}}^{n,\theta}\|^2
\end{equation}
Now this completes finding bounds for each term in the $RHS$ of (41). Putting common terms all together in the left hand side and multiplying them by $2$ and  then integrating both sides over $(t^n,t^{n+1})$ for $n=0,...,(N-1)$ , we  have 
\begin{multline}
\sum_{n=0}^{N-1} ( \|e^{n+1}_{\textbf{u}}\|^2-\|e^{n}_{\textbf{u}}\|^2)+ \{ 2\mu_l-  2(2+\epsilon_1') C_{\tau_2} - 2\mu C_P -\epsilon_2' \}  \sum_{n=0}^{N-1} \int_{t^n}^{t^{n+1}} \|e^{n,\theta}_{\textbf{u}}\|_1^2 dt \\
\leq h^2 \sum_{n=0}^{N-1} \int_{t^n}^{t^{n+1}} \{ \bar{C}_1 \|\textbf{R}_1^{n,\theta}\| + \bar{C}_2 \|R_2^{n,\theta}\|+ \frac{2 C_{\tau_2}}{\epsilon_1'} \bar{C}_5 + h^2 \epsilon_2' \bar{C}_5 \}dt+  \\
\quad   2 \frac{\mid \tau_1 \mid }{T_0- C_{\tau_1}} [ (T+C_{\tau_1}) (\sum_{k=1}^{n_{el}} D_{B_{1k}}) + 2  \sum_{i=1}^2 \sum_{k=1}^{n_{el}} B_{1k}^i ]  \sum_{n=0}^{N-1} \int_{t^n}^{t^{n+1}} \|\textbf{R}_1^{n,\theta}\|^2 dt  \\
+
\frac{2}{\epsilon_2'} \sum_{n=0}^{N-1} \int_{t^n}^{t^{n+1}}\|\textbf{TE}^{n,\theta}_1\|^2dt \hspace{40mm}
\end{multline}
Choose the arbitrary parameters including $C_{\tau_2}$ and the $Poincare$ constant $C_P$ in such a way that all the coefficients in the left hand side can be made positive.
Then taking minimum over the coefficients in the left hand side let us divide both sides by them. Using property (10) associated with both implicit time discretisation scheme  and the fact that $\tau_1$ is of order $h^2$, we have arrived at the following relation:
\begin{equation}
\boxed{\|e_{\textbf{u}}\|_{\tilde{\textbf{V}}}^2   \leq C'(\textbf{R}) (h^2+dt^{2r})}
\end{equation}
where
\begin{equation}
    r=
    \begin{cases}
      1, & \text{if}\ \theta=1 \hspace{1mm} for  \hspace{1mm} backward  \hspace{1mm} Euler  \hspace{1mm} rule \\
      2, & \text{if}\ \theta=0  \hspace{1mm} for  \hspace{1mm} Crank-Nicolson  \hspace{1mm} scheme
    \end{cases}
  \end{equation}
 This only completes one part of $aposteriori$ estimation and in the next part we combine the corresponding pressure part.\vspace{2mm}\\
\textbf{Second part}: By Galerkin orthogonality followed by property (I) of projection operator we have
\begin{equation}
b(\textbf{v}_h, I_hp-p_h)  = (\frac{\partial e_{\textbf{u}}}{\partial t}, \textbf{v}_{h})+ a_{S}(e_{\textbf{u}},\textbf{v}_h)
\end{equation}
Integrating both sides with respect time and later applying $Cauchy-Schwarz$'s inequality, $Young$'s inequality and the above result (53) we have
\begin{equation}
\begin{split}
\sum_{n=0}^{N-1} \int_{t^n}^{t^{n+1}} b(\textbf{v}_h, E_p^{A,n,\theta}) dt & = \sum_{n=0}^{N-1} \int_{t^n}^{t^{n+1}} \{ (\frac{e_{\textbf{u}}^{n+1}-e_{\textbf{u}}^n}{dt},\textbf{v}_{h})+ a_{S}(e_{\textbf{u}}^{n,\theta},\textbf{v}_h) \}dt\\
&  \leq \bar{C}'(\textbf{R})(h^2+ dt^{2r}) \|\textbf{v}_h\|_1
\end{split}
\end{equation} 
Now applying this result on the following relation
\begin{equation}
\begin{split}
\|I_hp-p_h\|_{L^2(L^2)}^2& = \|E_p^A\|_{L^2(L^2)}^2 \\
& = \sum_{n=0}^{N-1} \int_{t^n}^{t^{n+1}} \|E_p^{A,n,\theta}\|^2 dt \\
& \leq  \sum_{n=0}^{N-1} \int_{t^n}^{t^{n+1}} \underset{\textbf{v}_h}{sup} \frac{b(\textbf{v}_h, E_p^{A,n,\theta})}{\|\textbf{v}_h\|_1} dt \\
& \leq \bar{C}'(\textbf{R})(h^2+ dt^{2r})
\end{split}
\end{equation}
Now combining the results obtained in the first and second part and applying interpolation estimate on pressure interpolation term $E^I_p$, we finally arrive at the following $aposteriori$ error estimate:
\begin{equation}
\boxed{\|\textbf{u}-\textbf{u}_h\|^2_{\tilde{\textbf{V}}} + \|p-p_h\|_{L^2(L^2)}^2   \leq \bar{C}(\textbf{R}) (h^2+ dt^{2r})}
\end{equation} 
\end{proof}
\begin{remark}
These estimations clearly imply that the scheme is $first$ order convergent in space with respect to total norm, whereas in time it is $first$ order convergent for backward Euler time discretization scheme and $second$ order convergent for Crank-Nicolson method.
\end{remark}

\section{Numerical Experiment}
In this section we have numerically verified the convergence rate established theoretically  under stabilized method in the previous section. For simplicity we have considered bounded square domain $\Omega$= (0,1) $\times$ (0,1). We have taken continuous piecewise linear finite element(P1) space into account for approximating both the velocity and pressure variables and applied backward Euler time discretization rule.
 Let us mention here the exact solutions :
$\textbf{u}=(e^{-t}x^2(x-1)^2y(y-1)(2y-1), -e^{-t}y^2(y-1)^2x(x-1)(2x-1))$ 
and  $p=e^{-t}(2x-1)(2y-1)$.  \vspace{1mm} \\
The viscosity coefficients $\mu=0.1$
The stabilization parameters: $\tau_1=\frac{h^2}{4 \mu}$ and $\tau_2=\frac{2}{\tau_1}$  \\

 \begin{table}[]
\centering
\begin{tabular}{| *{6}{c|} }
    \hline
 Time &  Grid     & \multicolumn{2}{c|}{ASGS method}\\
      \hline       
  step & size     & Total error & RoC  \\
 \hline
0.1& 10 $\times$ 10 & 0.0651611 &   \\
   \hline
 0.05&    20 $\times$ 20 & 0.0349207 & 0.899928  \\  
     \hline
  0.025&   40 $\times$ 40 & 0.0182159 & 0.93883  \\   
    \hline         
 0.0125&   80 $\times$ 80 & 0.00931727 & 0.96722  \\  
     \hline         
 0.00625&   160 $\times$ 160 & 0.0047026 & 0.986447  \\    
    \hline      
\end{tabular}
\caption{Total error and Rate of convergence(RoC) under  $ASGS$ method for transient Stokes model at $T=1$}
    \end{table}

\begin{remark}
Table 1 is showing total error under $ASGS$ method at each mesh size and at different time steps.Clearly the order of convergence under $ASGS$ method is 1, which justifies theoretically established result.
\end{remark}

\section{Conclusion}
The paper presents $ASGS$ stabilized finite element analysis of  transient Stokes fluid flow model. Whereas this paper in one hand has elaborately derived the stabilized formulation, on other hand it has presented stability analysis of the stabilized formulation where the stabilized bilinear form is shown coercive and discrete $inf$-$sup$ condition has been established to ensure pressure stability. As well as $aposteriori$ error estimation has been carried out elaborately. It is essential to mention that the norm employed for error estimation consists of the full norms corresponding to each variable belonging to their respective spaces. Therefore it provides a wholesome information about convergence of the method. Theoretically the rate of convergence for $aposteriori$ error estimation turns out to be $O(h)$ in space for different time discretization rules. In numerical section the theoretically established result is well verified through a test case problem. This piece of work will definitely be useful for studying adaptivity and control of solution error. \vspace{1cm}\\


\begin{thebibliography}{}
\bibitem{RefA}
Y. Amanbek, M.F. Wheeler, A priori error analysis for transient problems using Enhanced Velocity approach in the discrete-time setting, Journal of Computational and Applied Mathematics 361, 459-471(2019).
\bibitem{RefB}
S. Badia, R. Codina, Unified stabilized finite element formulations for the Stokes and Darcy problems, SIAM J. Numer. Anal., 47(3), 1971–2000(2009).
\bibitem{RefC}
T.J.R. Hughes, Multiscale phenomena: Green’s functions, the Dirichlet to Neumann formulation, subgrid scale models, bubbles and the origins of stabilized methods, Comput. Methods Appl. Mech. Engrg. 127,387-401(1995) .
\bibitem{RefD}
R. Codina, Comparison of some finite element methods for solving the
diffusion-convection-reaction equation, Comput. Methods Appl. Mech. Engrg.
156, 185-210 (1998).
\bibitem{RefE}
J. Douglas Jr., J.P. Wang, An absolutely stabilized finite element method for the Stokes problem, Math. Comp. 52 (186), 495–508(1989)
\bibitem{RefF}
L. Tobiska, R. Verfurth, Analysis of a streamline diffusion finite element method for the Stokes and Navier–Stokes equations, SIAM J. Numer. Anal. 33 (1), 107–127(1996).
\bibitem{RefG}
C. Johnson, J. Saranen, Streamline diffusion methods for the incompressible Euler and Navier–Stokes equations, Math. Comp. 47 (175) , 1–18(1986).
\bibitem{RefH}
T.J.R. Hughes, L.P. Franca, M. Balestra, A new finite element formulation for computational fluid dynamics. V. Circumventing the Babuska–Brezzi condition: a stable Petrov–Galerkin formulation of the Stokes problem accommodating equal-order interpolations, Comput. Methods Appl. Mech. Engrg. 59 (1), 85–99. (1986).
\bibitem{RefI}
E. Burman, M. A. Fernández, Analysis of the PSPG method for the transient Stokes’ problem,Comput. Methods Appl. Mech. Engrg. 200, 2882-2890 (2011).
\bibitem{RefJ}
E. Burman, M. A. Fernández, Galerkin finite element methods with symmetric pressure stabilization for the transient Stokes equations: stability and convergence analysis, SIAM J. Numer. Anal. 47, 409–439 (2008) .
\bibitem{RefK}
N. Ahmed, S. Becher, G. Matthies, Higher-order discontinuous Galerkin time stepping and local
projection stabilization techniques for the transient Stokes problem, Comput. Methods Appl. Mech. Engrg. 313, 28-52 (2017).
\bibitem{RefD1}
V. Girault, P.A. Raviart, Finite Element Methods for Navier-Stokes Equations: Theory and Algorithms, vol. 5 of Springer series in computational mathematics. Springer, Berlin (1986).
\bibitem{RefD2}
E. Burman, M. A. Fernández, Continuous interior penalty finite element method for the time-dependent Navier–Stokes equations: space discretization and convergence, Numer. Math., 107,39–77(2007).
\bibitem{RefD3}
J. A. Wheeler, M. F. Wheeler, I. Yotovc, Enhanced velocity mixed finite element methods for
flow in multiblock domains, Computational Geosciences 6: 315–332(2002).
\bibitem{RefL}
B. Rivi$\grave{e}$re, M. F. Wheeler, A Discontinuous Galerkin Method Applied to Nonlinear Parabolic Equations, Discontinuous Galerkin Methods, Springer, pp. 231-244(2000).




\end{thebibliography}
\end{document}